\documentclass[a4paper,11pt,reqno]{amsart}
\usepackage{amsmath,amsfonts,amssymb,amsthm}
\usepackage{tikz}

\numberwithin{equation}{section}

\newtheorem{theorem}{Theorem}[section]
\newtheorem{lemma}[theorem]{Lemma}
\newtheorem{definition}[theorem]{Definition}

\newtheorem{remark}[theorem]{Remark}

\newtheorem{corollary}[theorem]{Corollary}

\def\R{\mathbb{R}}
\def\RR{\mathcal{R}}\def\N{\mathbb{N}}
\def\H{\mathcal H}
\def\eps{\varepsilon}

\def\S{\mathbb{S}}

\def\step#1#2{\par\noindent{\underline{\it Step~#1.}}\emph{ #2}\\}

\def\bal{\begin{aligned}}
\def\eal{\end{aligned}}

\title[Some estimates for the higher eigenvalues of sets close to the ball]{Some estimates for the higher eigenvalues\\of sets close to the ball}
\author{Dario Mazzoleni}\address{Dipartimento di Matematica F. Casorati, Universit\`a degli Studi di Pavia, Via Ferrata, 1 -- 27100 Pavia, Italy. {\tt dario.mazzoleni@unipv.it}}
\author{Aldo Pratelli}\address{Department Mathematik, Friedrich-Alexander Universit\"at Erlangen-N\"urnberg, Cauerstrasse, 11 -- 91058 Erlangen, Germany. {\tt pratelli@math.fau.de}}

\begin{document}

\begin{abstract}
In this paper we investigate the behavior of the eigenvalues of the Dirichlet Laplacian on sets in $\R^N$ whose first eigenvalue is close to the one of the ball with the same volume. In particular in our main Theorem~\ref{finalmain} we prove that, for all $k\in\N$, there is a positive constant $C=C(k,N)$ such that for every open set $\Omega\subseteq \R^N$ with unit measure and with $\lambda_1(\Omega)$ not excessively large one has
\begin{align*}
|\lambda_k(\Omega)-\lambda_k(B)|\leq C (\lambda_1(\Omega)-\lambda_1(B))^\beta\,, && \lambda_k(B)-\lambda_k(\Omega)\leq Cd(\Omega)^{\beta'}\,,
\end{align*}
where $d(\Omega)$ is the Fraenkel asymmetry of $\Omega$, and where $\beta$ and $\beta'$ are explicit exponents, not depending on $k$ nor on $N$; for the special case $N=2$, a better estimate holds.
\end{abstract}

\maketitle

\section{Introduction}
The study of eigenvalues of the Dirichlet Laplacian from the point of view of shape optimization and spectral geometry has been strongly investigated in recent years.
This is a very beautiful and interesting topic mostly because there are still a lot of rather basic facts that are very hard to prove, although they seem very ``natural''.

A classical shape optimization problem is the minimization of the principal frequency (i.e., the first Dirichlet eigenvalue) among sets of given volume. This problem was solved by Faber and Krahn in the 1920s, who independently proved that for every open set $\Omega\subseteq \R^N$ of unit measure one has $\lambda_1(\Omega)\geq \lambda_1(B)$, where $B$ is the ball of unit measure centered at the origin and $\lambda_k(\cdot)$ is the $k$-th eigenvalue of the Dirichlet Laplacian; in addition, the inequality is strict unless $\Omega$ itself is a ball. In more recent years there have been many studies concerning improvements of the Faber--Krahn inequality in a \emph{quantitative} direction. This accounts to understand whether sets that are close to optimality (i.e., with $\lambda_1(\Omega)-\lambda_1(B)\ll 1$) can be said to be ``close'' to the ball in some sense. A complete answer to this question was given by Brasco, De Philippis and Velichkov~\cite{bdpv} (after several other contributions on this topic, for which we refer to the paper itself, as well as to the survey~\cite{fusco}). They proved that, for every open set $\Omega\subseteq\R^N$ with unit measure, one has
\begin{equation}\label{bdv}
\lambda_1(\Omega)-\lambda_1(B) \geq \frac 1C\, d(\Omega)^2\,,
\end{equation}
where $d(\Omega)$ is the well-known \emph{Fraenkel asymmetry}, defined as
\begin{equation}\label{fraenas}
d(\Omega):= \inf \Big\{ \big|\Omega\Delta (x+B)\big|\,, x\in\R^N\Big\}\,.
\end{equation}
It is important to notice that the exponent $2$ in~(\ref{bdv}) is sharp.

In this paper we consider a set for which the difference $\lambda_1(\Omega)-\lambda_1(B)$ is very small, hence which is close to a suitable unit ball according to~(\ref{bdv}). We are interested in the following question: is it true that also the higher eigenvalues $\lambda_k(\Omega)$ are close to the corresponding eigenvalues of the unit ball? It is simple to guess that it must be so, our aim is to prove this with an estimate which goes like a power. More precisely, our main result is the following.

\begin{theorem}\label{finalmain}
Let $k,N\in \N$, and let $\Omega\subseteq \R^N$ be an open set of unit measure with $\lambda_1(\Omega)\leq \lambda_1(B)+1$. Then, for every $\eta>0$ there exists a constant $C=C(k,N,\eta)$ such that for $N\geq 3$
\begin{gather}
-C\, (\lambda_1(\Omega)-\lambda_1(B))^{\frac 16-\eta} \leq \lambda_k(\Omega)-\lambda_k(B)\leq C(\lambda_1(\Omega)-\lambda_1(B))^{\frac1{12}-\eta}\,, \label{mainest}\\
\lambda_k(B)-\lambda_k(\Omega)\leq Cd(\Omega)^{\frac13-\eta}\,, \label{estfraenkel}
\end{gather}
while for $N=2$ the exponents can be made better as follows,
\begin{gather}
-C\, (\lambda_1(\Omega)-\lambda_1(B))^{\frac 14-\eta} \leq \lambda_k(\Omega)-\lambda_k(B)\leq C(\lambda_1(\Omega)-\lambda_1(B))^{\frac18-\eta}\,, \tag{\ref{mainest}'}\label{mainest'}\\
\lambda_k(B)-\lambda_k(\Omega)\leq Cd(\Omega)^{\frac12-\eta}\,.\tag{\ref{estfraenkel}'}\label{estfraenkel'}
\end{gather}
\end{theorem}

Let us briefly comment on the result. For a set $\Omega$ with unit volume and $\lambda_1(\Omega)-\lambda_1(B)\ll 1$, the closeness of $\lambda_k(\Omega)$ to $\lambda_k(B)$ is not an obvious consequence of~(\ref{bdv}). Indeed, that estimate says that $\Omega$ has a small symmetric difference with a unit ball; nevertheless, a set with small symmetric difference with a unit ball could also have huge eigenvalues, for instance if there are several small holes in the set. On the bright side, it is not hard to guess that the effect of these holes is more important on the small eigenvalues, hence if there is only a small effect on the first eigenvalue, the same should be true also for the higher ones. This rough argument is more or less at the basis of our construction: more precisely, we will use the information that the first eigenvalue is only a little bigger than the one of the ball to infer that the same is true for the following eigenvalues; to do so, we will use the first eigenfunction to construct competitors to the other eigenfunctions.\par

Another comment has to be done about the sign of $\lambda_k(\Omega)-\lambda_k(B)$: indeed, by the Faber-Krahn inequality we know that $\lambda_1(\Omega)\geq \lambda_1(B)$, but there is no reason why also $\lambda_k(\Omega)$ should be bigger than $\lambda_k(B)$, on the contrary it can easily happen to be smaller. In fact, the estimate~(\ref{mainest}) is done by two inequalities, with different exponents, which need completely different approaches.\par

A last comment has to be done about the constant $C$ in Theorem~\ref{finalmain}: our main aim is to obtain a good power, while we are not very interested in the value of $C$. As a consequence, even if the constant is explicitely computable from our construction, we avoid to do so for simplicity; hence, as usual, the letter $C$ will be used for a geometrical constant, which can increase from line to line.\par

The plan of the paper is rather simple: in the short Section~\ref{sec:prel} below, we recall some basic facts about eivenvalues of the Dirichlet Laplacian and we fix the notation. Then, in Section~\ref{sec:b-om} we prove, mainly via a self-improving argument, the estimates~(\ref{estfraenkel}) and~(\ref{estfraenkel}'), from which the left inequality in~(\ref{mainest}) and~(\ref{mainest'}) immediately follows by~(\ref{bdv}). Finally, in Section~\ref{sec:om-b} we prove the right inequality in~(\ref{mainest}) and~(\ref{mainest'}).

\begin{remark}
Bertrand and Colbois, in~\cite{bc}, proved some quantitative estimates for the higher eigenvalues of sets close to the ball, mostly with capacitary methods.
Theorem~\ref{finalmain} improves their estimates~\cite[Th\'eor\`eme~3.3]{bc}, whose exponents depend on the dimension.
\end{remark}

\subsection{Notation and preliminaries\label{sec:prel}}
Throughout the paper, we work in the Euclidean space $\R^N$ and the generic point is written as $x=(x_1,\dots, x_N)\in\R^N$. Given an open set $\Omega\subseteq \R^N$ with finite Lebesgue measure, we denote by $\lambda_k(\Omega)$ its $k$--th eigenvalue of the Dirichlet Laplacian, counted with due multiplicity. A corresponding eigenfunction will be called $u_k$, and it will always be taken with unit  $L^2$ norm. It is important to recall that the eigenvalues can be characterized through a min-max principle, that is, for all $k\in \N$,
\begin{equation}\label{min-max}
\lambda_k(\Omega)=\min_{E_k\subseteq H^1_0(\Omega)}\max_{u\in E_k\setminus \{0\}}\frac{\int_\Omega|Du|^2}{\int_\Omega u^2}\,,
\end{equation}
where the minimum is taken over all the $k-$dimensional subspaces $E_k$ of $H^1_0(\Omega)$. Moreover, whenever $u\in H^1_0(\Omega)$ is given, we denote its so called \emph{Rayleigh quotient} as
\[
\RR(u):=\frac{\int_\Omega{|Du|^2}}{\int_\Omega u^2}\,.
\]
It is well-known, and it directly follows from the min-max principle, that the eigenvalues are monotone with respect to set inclusion, i.e. $\Omega_1\subseteq \Omega_2$ implies $\lambda_k(\Omega_1)\geq \lambda_k(\Omega_2)$. Moreover, the following scaling property holds,
\begin{equation}\label{scaling}
\lambda_k(t\Omega)=t^{-2}\lambda_k(\Omega)\,,\qquad \forall\,k\in\N,\;\;t>0,\;\;\Omega\subseteq \R^N\,.
\end{equation}

Another important property of eigenfunctions is that they are bounded in $L^{\infty}$, only depending on the corresponding eigenvalue; more precisely, one has
\begin{equation}\label{boundlinfty}
\|u_k\|_{L^{\infty}(\Omega)}\leq e^{1/8\pi}\lambda_k(\Omega)^{N/4}\,,
\end{equation}
for a proof see for instance~\cite[Example 2.1.8]{davies}.\par

It is also useful to keep in mind that higher eigenvalues cannot be too big if the first eigenvalue is small; more precisely, there exists a constant $M_k$, depending only on $k$ and on $N$, such that for every set $\Omega\subseteq\R^N$ one has
\begin{equation}\label{dario}
\frac{\lambda_k(\Omega)}{\lambda_1(\Omega)}\leq M_k\,;
\end{equation}
for a simple proof one can refer to~\cite[Theorem~B]{mp}, more involved estimates are proved in~\cite{A}. A special case is given by $k=2$, in which the ratio is actually maximied by the ball, that is,
\[
\frac{\lambda_2(\Omega)}{\lambda_1(\Omega)}\leq \frac{\lambda_2(B)}{\lambda_1(B)}\,,
\]
as proved by Ashbaugh and Benguria in~\cite{ab}. Notice that, as an immediate consequence, for any set $\Omega$ we can write
\[
\lambda_2(\Omega)-\lambda_2(B)\leq \frac{\lambda_2(B)}{\lambda_1(B)}\, \big(\lambda_1(\Omega)-\lambda_1(B)\big)\,,
\]
so the exponent $1/12-\eta$ or $1/8-\eta$ can be replaced by the exponent $1$ in the right inequality in~(\ref{mainest}) and~(\ref{mainest'}) if $k=2$.\par

We conclude with a last piece of notation. As already said, by $B$ we denote the unit ball centered at the origin, then with radius $\omega_N^{-1/N}$. For every small number $\gamma$, positive or negative, we will then denote by $B_\gamma$ the ball centered at the origin and with radius $\omega_N^{-1/N}+\gamma$, so that $B=B_0$. Notice that $|B_\gamma|=1 + N\omega_N^{1/N} \gamma + o(\gamma)$.\\
Finally, given two real numbers $a,b$, we denote minimum between $a$ and $b$ as $a\wedge b$.

\section{The proof of~(\ref{estfraenkel}) and~(\ref{estfraenkel'}), and the estimate of $\lambda_k(\Omega)-\lambda_k(B)$ from below\label{sec:b-om}}

This section is devoted to prove the estimates~(\ref{estfraenkel}) and~(\ref{estfraenkel'}). Notice that, once having these inequalities, the left inequalities in~(\ref{mainest}) and~(\ref{mainest'}) are obvious by~(\ref{bdv}). Our idea is to try to reduce ourself to the case of a set $\Omega$ contained in a ball $B_{d^\alpha}$, where $d=d(\Omega)$ is the Fraenkel asymmetry defined in~(\ref{fraenas}), while $\alpha>0$ is a suitable power. The reason to do so is clear: if $\Omega\subseteq B_{d^\alpha}$, then
\[
\lambda_k(\Omega)\geq \lambda_k(B_{d^\alpha}) \geq \lambda_k(B) - C d^\alpha\,,
\]
thus the searched estimate immediately follows (with exponent $\alpha$). The main tool we will use to reduce us to this simple case will be the following lemma, which will allow us to perform a bootstrap argument. 
\begin{lemma}\label{lemmagen}
Let $i,k\in \N$, let $\alpha>0$, let $\Omega\subseteq\R^N$ be a set of unit measure satisfying $\lambda_1(\Omega)\leq \lambda_1(B)+1$, and let $C_*>0$ and $\gamma\geq 0$ be so that for each $1\leq j \leq k$ it is
\begin{equation}\label{condgamma}
\int_{\Omega\setminus B_{i\eps^\alpha}} |D u_j|^2\,dx\leq C_* \eps^{\gamma}\,,
\end{equation}
being $\eps=|\Omega\Delta B|$. Then, there exists some $\bar t\in \big(i\eps^\alpha,(i+1)\eps^\alpha\big)$ such that
\begin{align}\label{propgamma}
\int_{S_{\bar t}}|Du_j|^2\,d\H^{N-1}\leq C'\eps^{\gamma-\alpha}\,, &&
\H^{N-1}(S_{\bar t})\leq C'\eps^{1-\alpha}\,, &&
\int_{S_{\bar t}} u_j^2\,d\H^{N-1}\leq C' \eps^\theta\,,
\end{align}
for all $j=1,\dots, k$, where $S_t=\Omega\cap \partial B_t$, $C'$ is a constant depending only on $C_*$ and where
\begin{equation}\label{deftheta}
\theta=\left\{ \begin{array}{ll}
2-3\alpha+\gamma &\hbox{for $N=2$}\,,\\
1-\alpha &\hbox{for $N\geq 3$}\,.
\end{array}\right.
\end{equation}
Moreover, for every $1\leq j\leq k$ one also has
\begin{equation}\label{bootstrap}
\bigg|\int_{\R^N\setminus B_{\bar t}}|D u_j|^2\,dx\bigg|\leq C\Big(\eps+\eps^{\frac{\gamma-\alpha+\theta}2}\Big)\,.
\end{equation}
\end{lemma}
\begin{proof}
Let us denote for brevity by $I$ the interval $\big(i\eps^\alpha,(i+1)\eps^\alpha\big)$, which has length $\eps^\alpha$. By~(\ref{condgamma}), for every $1\leq j\leq k$ we have
\[
C_* \eps^\gamma \geq \int_{\R^N\setminus (\Omega \cap B_{i\eps^\alpha})} |D u_j|^2\,dx
\geq \int_{t\in I} \int_{S_t} |Du_j|^2\, d\H^{N-1}\, dt\,,
\]
then the set $I_1\subseteq I$ of all those $t\in I$ such that
\[
\int_{S_t} |Du_j|^2 \, d\H^{N-1}\leq 3 k C_* \eps^{\gamma-\alpha}
\]
for every $1\leq j \leq k$ has length at least $2/3\, \eps^\alpha$. Then, any $\bar t\in I_1$ satisfies the first inequality in~(\ref{propgamma}), up to choose $C'\geq 3k C_*$.\par

In addition, since
\[
\eps = |\Omega\Delta B| \geq \int_{t\in I_1} \H^{N-1}( S_t) \, dt\,,
\]
then the subset $I_2$ of $I_1$ of all those $t$ such that $\H^{N-1}(S_t)\leq 3 \eps^{1-\alpha}$ has measure at least $\eps^\alpha/3$. Notice that any $\bar t\in I_2$ satisfies also the second inequality in~(\ref{propgamma}), as soon as $C'\geq 3$.\par

Recalling the rescaling properties~(\ref{scaling}) of the eigenvalues, as well as the Faber-Krahn inequality, there exists a geometrical constant $C_N$ such that, for every $t\in I_2$, one has
\[
\lambda_1(S_t) \geq C_N \H^{N-1}(S_t)^{-\frac 2{N-1}}
\geq C_N (3\eps^{1-\alpha})^{-\frac 2{N-1}} = \widetilde C_N \eps^{-\frac{2(1-\alpha)}{N-1}}\,.
\]
As a consequence, again by~(\ref{condgamma}) we have
\[\begin{split}
C_* \eps^{\gamma} &\geq \int_{t\in I_2} \int_{S_t} |D u_j|^2\, d\H^{N-1}\,dt
\geq \int_{t\in I_2} \lambda_1(S_t) \int_{S_t} u_j^2\, d\H^{N-1}\,dt\\
&\geq \widetilde C_N \eps^{-\frac{2(1-\alpha)}{N-1}} \int_{t\in I_2} \int_{S_t} u_j^2\, d\H^{N-1}\,dt\,,
\end{split}\]
which ensures the existence of some $\bar t\in I_2$ such that
\[
\int_{S_{\bar t}} u_j^2 \leq 3\, \frac {C_*}{\widetilde C_N}\, k \eps^{\frac{2(1-\alpha)}{N-1}+\gamma-\alpha}
\leq C'\eps^{\frac{2-(N+1)\alpha}{(N-1)}+\gamma}
\]
for every $1\leq j\leq k$, where the last inequality holds up to define $C'\geq 3 k C_*/\widetilde C_N$. Hence, any such $\bar t$ satisfies also the last estimate in~(\ref{propgamma}) if $\theta=\frac{2-(N+1)\alpha}{(N-1)}+\gamma$: notice that this number equals $2-3\alpha+\gamma$ if $N=2$, so we have concluded~(\ref{propgamma}) if $N=2$. Moreover, by~(\ref{boundlinfty}) and~(\ref{dario}), and since $\lambda_1(\Omega)\leq \lambda_1(B)+1$, we have
\begin{equation}\label{stiminf}
\|u_j\|_{L^\infty} \leq C \lambda_j(\Omega)^{N/4} \leq C \lambda_k(\Omega)^{N/4} \leq C (M_k\lambda_1(\Omega))^{N/4} \leq C
\end{equation}
for some constant $C$ only depending on $k$ and $N$. Then, for every $\bar t\in I_2$ one directly has
\[
\int_{S_{\bar t}} u_j^2\leq C \H^{N-1}(S_{\bar t})\leq C' \eps^{1-\alpha}\,,
\]
up to modify $C'$ in the obvious way, so the last estimate in~(\ref{propgamma}) holds with the choice $\theta=1-\alpha$, which concludes also the case $N\geq 3$.\par
Finally, to obtain~(\ref{bootstrap}) and hence conclude the proof, it is enough to recall that $-\Delta u_j = \lambda_j(\Omega) u_j$ in $\Omega$ and to use Divergence Theorem, (\ref{stiminf}), H\"older inequality, and~(\ref{propgamma}), getting
\[\begin{split}
\bigg|\int_{\R^N\setminus B_{\bar t}} |Du_j|^2\, dx\bigg|
&\leq \lambda_j \bigg|\int_{\R^N\setminus B_{\bar t}} u_j^2\, dx\bigg| + \int_{S_{\bar t}} |Du_j| |u_j| \, d\H^{N-1}\\
&\leq C\eps +\sqrt{\int_{S_{\bar t}} |Du_j|^2}\sqrt{\int_{S_{\bar t}} u_j^2}
\leq C\eps + C\eps^{\frac{\gamma-\alpha+\theta}2}\,.
\end{split}\]
\end{proof}

\begin{remark}
As it appears clear from the above proof, the last inequality of~(\ref{propgamma}) holds also with $\theta=1-\alpha$ when $N=2$, and it holds also with $\theta=\frac{2-(N+1)\alpha}{(N-1)}+\gamma$ when $N\geq 3$. Nevertheless, for the exponents $\alpha$ which will be used in this paper, the choice of $\theta$ done in~(\ref{deftheta}) is the best one.
\end{remark}

\begin{corollary}\label{maincor}
For every $\alpha<1/(3\wedge N)$ and every $k,\, N\in\N$ there exist $n\in \N$ and $C>0$ such that, for every set $\Omega\subseteq\R^N$ of unit measure with $\lambda_1(\Omega)\leq \lambda_1(B)+1$ and $\eps=|\Omega\Delta B|<1$, there is $\bar t < n\eps^\alpha$ such that, for every $1\leq j\leq k$,
\begin{align}\label{finalesti}
\int_{S_{\bar t}} |Du_j|^2 \leq C\,, &&
\int_{S_{\bar t}} u_j^2 \leq C \eps^{(1\vee (4-N))(1-\alpha)}\,, &&
\int_{\Omega\setminus B_{\bar t}} |Du_j|^2 \leq C \eps^\alpha\,.
\end{align}
\end{corollary}
\begin{proof}
First of all, we notice that it is enough to find some $n\in\N$ and a constant $C_0$, both depending only on $N,\,k$ and $\alpha$, so that
\begin{equation}\label{enough}
\int_{\Omega\setminus B_{(n-1)\eps^\alpha}} |Du_j|^2 \leq C_0\eps^\alpha\,.
\end{equation}
Indeed, in this case we have the validity of~(\ref{condgamma}) with $\gamma=\alpha$ and $C_*=C_0$, then Lemma~\ref{lemmagen} provides us with some $\bar t\in \big((n-1)\eps^\alpha,n\eps^\alpha\big)$ satisfying~(\ref{propgamma}). The first estimate in~(\ref{finalesti}) is directly given by the first estimate in~(\ref{propgamma}). Concerning the second one, it follows from the third estimate in~(\ref{propgamma}): indeed, if $N\geq 3$ we have to show that $\int_{S_{\bar t}} u_j^2\leq C\eps^{1-\alpha}$, which is true since $\theta=1-\alpha$. Instead, for $N=2$, we have to show that $\int_{S_{\bar t}} u_j^2 \leq C\eps^{2-2\alpha}$, which is again true because $\theta=2-2\alpha$. Finally, the third estimate in~(\ref{finalesti}) follows by~(\ref{enough}) since $\bar t> (n-1)\eps^\alpha$. Summarizing, to conclude the thesis we only have to show~(\ref{enough}).\par

Let us start with the case $N=2$. Recall that
\begin{equation}\label{obvious}
\int_{\Omega\setminus B} |Du_j|^2 \, dx \leq \int_\Omega |Du_j|^2\, dx = \lambda_j(\Omega) \leq \lambda_k(B)\,,
\end{equation}
hence~(\ref{condgamma}) clearly holds for every $1\leq j\leq k$ with $i=0$, $\gamma_0=0$ and $C_*=\lambda_k(B)$. We can then apply Lemma~\ref{lemmagen} to obtain some $\bar t_0\in (0,\eps^\alpha)$ such that~(\ref{propgamma}) and~(\ref{bootstrap}) hold. In particular, recalling that $\theta=2-3\alpha$, (\ref{bootstrap}) gives
\[
\int_{\Omega\setminus B_{\eps^\alpha}} |Du_j|^2 \leq 
\int_{\Omega\setminus B_{\bar t_0}} |Du_j|^2 \leq 
C \big(\eps + \eps^{1-2\alpha}\big)\leq C\eps^{1-2\alpha}\,,
\]
hence we have already found~(\ref{enough}) with $n=2$ if $1-2\alpha\geq \alpha$. Otherwise, the last inequality implies anyway that~(\ref{condgamma}) holds with $i=1$ and $\gamma_1=1-2\alpha$, hence again Lemma~\ref{lemmagen} provides us with some $\bar t_1\in (\eps^\alpha,2\eps^\alpha)$ satisfying~(\ref{propgamma}) and~(\ref{bootstrap}), in particular $\theta=2-3\alpha + \gamma_1$ and
\[
\int_{\Omega\setminus B_{2\eps^\alpha}} |Du_j|^2 \leq \int_{\Omega\setminus B_{\bar t_1}} |Du_j|^2
\leq C \big(\eps + \eps^{1-2\alpha+\gamma_1}\big)\leq C\eps^{2(1-2\alpha)}\,.
\]
Again, if $2(1-2\alpha)\geq \alpha$, then we have obtained~(\ref{enough}) with $n=3$; and otherwise, (\ref{condgamma}) holds true with $i=2$ and $\gamma_2=2(1-2\alpha)$. Continuing with the obvious recursion argument, we get the validity of~(\ref{condgamma}) with $\gamma_i=i(1-2\alpha)$ for every $i$ such that $i(1-2\alpha)<\alpha$; hence, we call $n$ the smallest integer such that $(n-1)(1-2\alpha)\geq\alpha$, thus~(\ref{enough}) is established. Notice that the final constant $C_0$ depends also on $n$, then actually on $\alpha$. The case $N=2$ is then concluded.\par

Let us now consider the case $N\geq 3$, which is slightly trickier. Taken any $\alpha<1/3$, we notice that the trivial estimate~(\ref{obvious}) still holds, hence we have again the validity of~(\ref{condgamma}) with $i=0$ and $\gamma_0=0$. Lemma~\ref{lemmagen} provides then some $\bar t_0 \in (0,\eps^\alpha)$ satisfying~(\ref{propgamma}) and~(\ref{bootstrap}). In particular, since $\theta=1-\alpha$ we have for every $1\leq j\leq k$ that
\[
\int_{\Omega\setminus B_{\eps^\alpha}} |Du_j|^2 \leq \int_{\Omega\setminus B_{\bar t_0}} |Du_j|^2 \leq C \Big(\eps+\eps^{\frac 12-\alpha}\Big) \leq C \eps^{\frac 12 - \alpha}\,,
\]
so~(\ref{enough}) already follows with $n=2$ if $\frac 12 - \alpha\geq \alpha$. Otherwise, (\ref{condgamma}) holds with $i=1$ and $\gamma_1=\frac 12 - \alpha$, which by Lemma~\ref{lemmagen} and again using that $\theta=1-\alpha$ gives some $\bar t_1\in (\eps^\alpha,2\eps^\alpha)$ satisfying~\eqref{bootstrap}. Thus we have
\[
\int_{\Omega\setminus B_{2\eps^\alpha}} |Du_j|^2 \leq \int_{\Omega\setminus B_{\bar t_1}} |Du_j|^2 \leq C \Big(\eps+\eps^{\frac 12-\alpha+\frac {\gamma_1}2}\Big) \leq C \eps^{\frac 32\big(\frac 12 - \alpha\big)}\,,
\]
so~(\ref{condgamma}) holds also with $i=2$ and $\gamma_2=\frac 32\big(\frac 12 - \alpha\big)$, and~(\ref{enough}) follows with $n=3$ if $\gamma_2\geq \alpha$. This time, the obvious recursion implies that~(\ref{condgamma}) holds with any $i\in\N$, where $\gamma_i$ is recursively defined by $\gamma_i=\frac 12 - \alpha + \frac{\gamma_{i-1}}2$, hence
\[
\gamma_i = \frac{2^i-1}{2^{i-1}}\,\bigg(\frac 12 - \alpha\bigg)\,.
\]
Notice that the sequence $\gamma_i$ is monotone increasing, and it converges to $1-2\alpha >\alpha$, where the last inequality is true since $\alpha<1/3$. Hence, we can define $n$ the smallest integer such that $\gamma_{n-1}\geq \alpha$, and we get then the validity of~(\ref{enough}) which, as already noticed, is enough to conclude.
\end{proof}

In the sequel, we will make use of the following definition.

\begin{definition}\label{heab}
Let $\Omega\subseteq \R^N$, $\bar t>0$, and let $\delta>0$ be given. Then, writing the generic point $0\neq x\in\R^N$ in polar coordinates as $x=\rho\theta$, with $(\rho,\theta)\in \R^+\times \S^{N-1}$, we define the sets
\begin{align*}
Q(\bar t,\delta):= \big\{(\rho,\theta)\in \R^+\times \S^{N-1}:\, (\bar t,\theta)\in S_{\bar t},\, \bar t\leq \rho\leq \bar t+\delta\big\}\,, &&
\widehat\Omega=:\Omega \cap B_{\bar t} \cup Q(\bar t,\delta)\,,
\end{align*}
and the functions
\[
\hat u_j(\rho\theta) := \left\{\begin{array}{ll}
u(\rho\theta) &\hbox{if $(\rho,\theta)\in \Omega\cap B_{\bar t}$}\,, \\
u(\bar t\theta)\bigg(1 - \bal\frac{\rho-\bar t}\delta\bigg)\eal\qquad &\hbox{if $(\rho,\theta)\in Q(\bar t,\delta)$}\,.
\end{array}\right.
\]
Notice that $\hat u_j\in H^1_0(\widehat\Omega)$.
\end{definition}

We observe now that $\lambda_j(\widehat\Omega)$ can be bounded by $\lambda_j(\Omega)$ plus a suitable power of $\eps$.
\begin{lemma}\label{lastone}
For every $N,\,k\in\N$ and $\alpha<1/(3\wedge N)$ there exist two constants $\bar\eps\ll 1$ and $C>0$ such that, if $\Omega\subseteq \R^N$ is an open set of unit measure with $\lambda_1(\Omega)\leq \lambda_1(B)+1$ and $\eps=|\Omega\Delta B|<\bar\eps$, and $\widehat\Omega$ is as in Definition~\ref{heab} with $\delta=\eps^\alpha$ and $\bar t$ given by Corollary~\ref{maincor}, then one has
\begin{equation}\label{estlamjhat}
\lambda_k(\widehat\Omega)\leq \lambda_k(\Omega) +C\eps^\alpha\,.
\end{equation}
\end{lemma}
\begin{proof}
For the sake of simplicity, we divide the proof in few steps.
\step{I}{Estimates about $\hat u_j$ on $Q(\bar t,\delta)$.}
First of all, we take any $1\leq j \leq k$ and we give the estimates about $\hat u_j$ on $Q(\bar t,\delta)$. The $L^2$ norm of $\hat u_j$ on $Q(\bar t,\delta)$ can be easily estimated thanks to~(\ref{finalesti}) as
\begin{equation}\label{L^2boundhatuj}
\int_{Q(\bar t, \delta)} \hat u_j^2\,dx\leq \delta \int_{S_{\bar t }}u_j^2\, d\H^{N-1}\leq C \eps\,.
\end{equation}
Concerning the $L^2$ norm of $D \hat u_j$ on $Q(\bar t,\delta)$, it is simpler to estimate separately the tangential and the radial part. Using again~(\ref{finalesti}), for the first one we get
\[
\int_{Q(\bar t, \delta)} |D_\tau \hat u_j|^2\,dx\leq \delta \int_{S_{\bar t }} |D_\tau u_j|^2\, d\H^{N-1}
\leq \delta \int_{S_{\bar t }} |D u_j|^2\, d\H^{N-1}
\leq C\eps^\alpha\,,
\]
while for the second one it is
\[
\int_{Q(\bar t, \delta)} |D_\rho \hat u_j|^2\,dx
\leq \frac 2\delta \int_{S_{\bar t }} u_j^2\, d\H^{N-1}
\leq C \eps^{(1\vee (4-N))(1-\alpha)-\alpha} \leq C\eps^\alpha\,,
\]
recalling that $\alpha<1/2$ if $N=2$ and $\alpha<1/3$ if $N\geq 3$. Hence, putting the two estimates together we have
\begin{equation}\label{hatujQtd}
\int_{Q(\bar t, \delta)} |D\hat u_j|^2\,dx \leq C\eps^\alpha\,.
\end{equation}
\step{II}{Estimates about $\RR(\hat u_j)$.}
In this step, we prove that for every $1\leq j \leq k$ one has
\begin{equation}\label{boundRRhatuj}
\RR(\hat u_j) \leq \lambda_j(\Omega)+C\eps^\alpha\,,
\end{equation}
while for every $m\neq j\in \{1,\, 2,\, \dots\, ,\, k\}$ it is
\begin{equation}\label{boundortho}
\bigg|\int_{\widehat\Omega} D\hat u_j \cdot D\hat u_m \bigg| + \bigg|\int_{\widehat\Omega}\hat u_j \hat u_m \bigg| \leq  C\eps^\alpha\,.
\end{equation}
In fact, by~(\ref{stiminf}) we have that
\begin{equation}\label{ujminusBt}
\int_{\Omega\setminus B_{\bar t}} u_j^2 \leq C |\R^N\setminus B_{\bar t}| \leq C\eps\,.
\end{equation}
Then, since $-\Delta u_j=\lambda_j(\Omega)u_j$ on $\Omega$ while $u_j=\hat u_j$ on $\Omega\cap B_{\bar t}$, and by (\ref{hatujQtd}) and~(\ref{ujminusBt}), we have
\[
\RR(\hat u_j)
=\frac{\bal\int_{\Omega\cap B_{\bar t}} |D u_j|^2+\int_{Q(\bar t,\delta)} |D\hat u_j|^2\eal}{\bal\int_{\Omega\cap B_{\bar t}} u_j^2 + \int_{Q(\bar t,\delta)} \hat u_j^2\eal}
\leq \frac{\bal\int_\Omega |Du_j|^2+C\eps^\alpha\eal}{\bal\int_{\Omega\cap B_{\bar t}} u_j^2\eal}
=\frac{\bal\lambda_j(\Omega)+C\eps^\alpha\eal}{\bal 1-\int_{\Omega\setminus B_{\bar t}} u_j^2\eal} \leq 
\lambda_j(\Omega) + C\eps^\alpha
\]
if $\bar\eps$ is small enough, hence~(\ref{boundRRhatuj}) already follows. To obtain~(\ref{boundortho}), instead, it is enough to recall that $u_j$ and $u_m$ are two orthogonal eigenfunctions in $\Omega$, so they are orthogonal both in $L^2(\Omega)$ and in $H^1(\Omega)$. Then, we can evaluate
\[\begin{split}
\bigg|\int_{\widehat \Omega} \hat u_j \hat u_m\bigg| &=
\bigg|\int_{\Omega\cap B_{\bar t}}  u_j  u_m + \int_{Q(\bar t,\delta)} \hat u_j \hat u_m\bigg|
=\bigg|\int_{\Omega\setminus B_{\bar t}}  u_j  u_m - \int_{Q(\bar t,\delta)} \hat u_j \hat u_m\bigg|\\
&\leq \sqrt{\int_{\Omega\setminus B_{\bar t}} u_j^2}\sqrt{\int_{\Omega\setminus B_{\bar t}} u_m^2}+\sqrt{\int_{Q(\bar t,\delta)} \hat u_j^2}\sqrt{\int_{Q(\bar t,\delta)} \hat u_m^2}\leq C\eps
\end{split}\]
by H\"older inequality, (\ref{ujminusBt}) and~(\ref{L^2boundhatuj}), and similarly
\[\begin{split}
\bigg|\int_{\widehat\Omega} D\hat u_j \cdot D\hat u_m \bigg| &=
\bigg|\int_{\Omega\setminus B_{\bar t}} D u_j \cdot D u_m -\int_{Q(\bar t,\delta)} D\hat u_j \cdot D\hat u_m\bigg|\\
&\leq \sqrt{\int_{\Omega\setminus B_{\bar t}} |D u_j|^2}\sqrt{\int_{\Omega\setminus B_{\bar t}} |D u_m|^2}+\sqrt{\int_{Q(\bar t,\delta)} |D \hat u_j|^2}\sqrt{\int_{Q(\bar t,\delta)} |D \hat u_m|^2}\leq C\eps^\alpha
\end{split}\]
by~(\ref{finalesti}) and~(\ref{hatujQtd}). From the last two bounds, (\ref{boundortho}) follows.
\step{III}{The space $E_k={\rm Span}\{\hat u_1,\, \hat u_2,\, \dots\, ,\, \hat u_k\}\subseteq H^1_0(\widehat \Omega)$ has dimension $k$.}
Let us now prove that the space $E_k$, generated by the functions $\hat u_m$ with $1\leq m\leq k$, has maximal dimension in $H^1_0(\widehat \Omega)$. Were it not so, there would be some $1\leq l \leq k$ and some numbers $\beta_m\in [-1,1]$ for every $1\leq m\leq k,\, m\neq l$, such that 
\[
\hat u_l = \sum_{1\leq m\leq k,\, m\neq l} \beta_m \hat u_m\,.
\]
As a consequence, by~(\ref{boundortho}) it would be
\[
\int_{\widehat\Omega} \hat u_l^2 = \sum_{1\leq m\leq k,\, m\neq l} \beta_m \int_{\widehat\Omega} \hat u_m \hat u_l
\leq \sum_{1\leq m\leq k,\, m\neq l} \bigg|\int_{\widehat\Omega} \hat u_m \hat u_l\bigg| \leq C\eps^\alpha\,,
\]
which, as soon as $\bar\eps\ll 1$, is in contradiction with the fact that, by~(\ref{ujminusBt}),
\begin{equation}\label{usetwice}
\int_{\widehat\Omega} \hat u_l^2 \geq  \int_{\Omega\cap B_{\bar t}} u_l^2 
= 1-\int_{\Omega\setminus B_{\bar t}} u_l^2 \geq  1 - C\eps\,.
\end{equation}
\step{IV}{Conclusion.}
We are now ready to estimate $\lambda_k(\widehat\Omega)$. Let $0\neq u\in E_k$ be a generic function; up to rescaling, we can write $u=\sum_{m=1}^k \beta_m \hat u_m$ with $\max\{ |\beta_m|,\, 1\leq m\leq k\}=1$. We can now evaluate
\[\begin{split}
\RR(u)=\frac{\bal\int_{\widehat\Omega} \Big|\sum_{m=1}^k \beta_m D\hat u_m\Big|^2\eal}{\bal\int_{\widehat\Omega} \Big(\sum_{m=1}^k \beta_m \hat u_m\Big)^2\eal}
\leq\frac{\bal \sum_m \beta_m^2 \int_{\widehat\Omega} |D\hat u_m|^2+\sum_{m\neq j} \bigg|\int_{\widehat\Omega} D\hat u_m \cdot D\hat u_j\bigg|\eal}
{\bal \sum_m \beta_m^2 \int_{\widehat\Omega} \hat u_m^2-\sum_{m\neq j} \bigg|\int_{\widehat\Omega} \hat u_m \hat u_j\bigg|\eal} =: \frac{K_1+K_2}{K_3-K_4}\,.
\end{split}\]
Keeping in mind~(\ref{boundRRhatuj}), we have
\[
\frac{K_1}{K_3} = \frac{\bal\sum_m \beta_m^2 \RR(\hat u_m) \int_{\widehat\Omega} \hat u_m^2\eal}{\bal \sum_m \beta_m^2 \int_{\widehat\Omega} \hat u_m^2\eal}
\leq \frac{\bal\sum_m \beta_m^2 \Big(\lambda_m(\Omega)+C\eps^\alpha\Big) \int_{\widehat\Omega} \hat u_m^2\eal}{\bal \sum_m \beta_m^2 \int_{\widehat\Omega} \hat u_m^2\eal}\leq \lambda_k(\Omega)+C\eps^\alpha\,.
\]
Moreover, by~(\ref{boundortho}) we have that $|K_2|+|K_4|\leq C\eps^\alpha$. Finally, (\ref{usetwice}) gives $K_3\geq 1-C\eps$. These estimates ensure that $\RR(u)=\frac{K_1+K_2}{K_3+K_4}\leq \lambda_k(\Omega)+C\eps^\alpha$. Since $0\neq u\in E_k$ was a generic function, thanks to Step~III and to the min-max principle~(\ref{min-max}) we obtain the validity of~(\ref{estlamjhat}).\end{proof}

The proof of~(\ref{estfraenkel}) follows now as an easy consequence.

\begin{proof}[Proof of Theorem~\ref{finalmain}, first part]
Let $\Omega\subseteq\R^N$ be any set of unit measure. We will prove now the validity of~(\ref{estfraenkel}) and~(\ref{estfraenkel'}), from which the left inequality in~(\ref{mainest}) and~(\ref{mainest'}) directly follows by~(\ref{bdv}). Up to a translation, we can assume that $\eps=|\Omega\Delta B|$ coincides with $d(\Omega)$. If $\lambda_k(\Omega)\geq \lambda_k(B)$, there is nothing to prove, so we can assume that $\lambda_k(\Omega)<\lambda_k(B)$. Take then any $\alpha<1/2$, if $N=2$, or $\alpha<1/3$, if $N\geq 3$, and let $\bar\eps$ be given by Lemma~\ref{lastone}. We can assume that $\eps<\bar\eps$, because otherwise~(\ref{estfraenkel}) is emptily true, up to increase the constant $C$ if necessary. We can then apply Lemma~\ref{lastone} to the set $\widehat\Omega$ given by Definition~\ref{heab} with $\delta=\eps^\alpha$, and we get the validity of~(\ref{estlamjhat}). Moreover, by construction we have that $\widehat\Omega\subseteq B_{\bar t+\eps^\alpha} \subseteq B_{(n+1)\eps^\alpha}$, where $n=n(k,N,\alpha)$ is the integer given by Corollary~(\ref{maincor}). Hence,
\[
\lambda_k(\widehat\Omega)\geq \lambda_k\big(B_{(n+1)\eps^\alpha}\big) \geq \lambda_k(B) - C\eps^\alpha\,,
\]
so that~(\ref{estfraenkel}) and~(\ref{estfraenkel'}) follow from~(\ref{estlamjhat}) and the proof is concluded.
\end{proof}

\section{The estimate of $\lambda_k(\Omega)-\lambda_k(B)$ from above\label{sec:om-b}}

In this section we prove the right inequality in~(\ref{mainest}) and~(\ref{mainest'}), thus concluding the proof of our main result. Notice that we cannot bound $\lambda_k(\Omega)-\lambda_k(B)$ by some power of the Fraenkel asymmetry; in fact, it is easy to build sets $\Omega$ of unit measure with arbitrarily small asymmetry $d(\Omega)$ and arbitrarly large eigenvalue $\lambda_k(\Omega)$. Observe also that we need an upper bound for $\lambda_k(\Omega)$; to get it, we have to find functions in $H^1_0(\Omega)$ which are almost orthogonal and whose Rayleigh quotients are close to the eigenvalues of the ball; of course one has to define these functions starting from the eigenfunctions of the ball, nevertheless these functions must vanish outside $\Omega$, so in particular in $B\setminus\Omega$, and this is clearly false for the eigenfunctions of the ball. The next lemma, which is the main ingredient of this step, will solve this issue.

\begin{lemma}\label{leminclus}
For every $k,\, N\in\N$ there exists a constant $C$ such that, for any open set $E\subseteq B\subseteq \R^N$ with $\lambda_1(E)\leq \lambda_1(B)+1$, one has
\begin{equation}\label{estcontenuto1}
\lambda_k(E)-\lambda_k(B)\leq C\sqrt{\lambda_1( E)-\lambda_1(B)}\,.
\end{equation} 
\end{lemma} 
\begin{proof}
Let us call $\{v_i\}_{i\in\N}$ a basis of eigenfunctions for $B$. Keep in mind that each $v_i$ is Lipschitz up to boundary, and $v_1$ is also radial and with growth at least linear away from the boundary, that is, there exists a constant $c$ such that $v_1(x) \geq c \,{\rm dist}(x,\partial B)$. Hence, it is easily seen that there exists a constant $K$, only depending on $N$ and on $k$, such that
\begin{equation}\label{lipbound}
\bigg\|\,\frac{v_j}{v_1}\,\bigg\|_{W^{1,\infty}(B)} \leq K\,.
\end{equation}
Let now $u_1$ be a first eigenfunction for $E$. Since $E\subseteq B$, then $u_1\in H^1_0(E)\subseteq H^1_0(B)$, thus we can write $u_1$ as
\[
u_1 = \alpha_1 v_1 + \sum_{i\geq 2} \alpha_i v_i\,,
\]
and since all the eigenfunctions have unit $L^2$ norm we have $\delta^2:= \sum_{i\geq 2} \alpha_i^2 = 1 - \alpha_1^2$. Notice that
\begin{equation}\label{stimadelta}
\lambda_1(E) = \alpha_1^2 \lambda_1(B) +\sum_{i\geq 2} \alpha_i^2\lambda_i(B)=\lambda_1(B) +\sum_{i\geq 2} \alpha_i^2(\lambda_i(B)-\lambda_1(B))\geq \lambda_1(B) + C\delta^2\,.
\end{equation}
Assuming without loss of generality that $\alpha_1>0$ and calling $\varphi=v_1-u_1=(1-\alpha_1) v_1 -\sum_{i\geq 2} \alpha_i v_i$, we have
\begin{equation}\label{stf}
\|\varphi\|_{H^1(B)} \leq C\delta\,.
\end{equation}
Let us now define the functions
\begin{equation}\label{deftildeuj}
\tilde u_j :=\frac{v_j}{v_1}\, u_1= v_j -\frac{v_j}{v_1}\,\varphi\,, \qquad \forall\, 1\leq j \leq k\,,
\end{equation}
and notice that each $\tilde u_j$ belongs to $H^1_0(E)$ thanks to~(\ref{lipbound}), and that $\tilde u_1=u_1$. The key observation is that, roughly speaking, if $\lambda_1(E)\approx \lambda_1(B)$ then it must be $u_1\approx v_1$, hence $\tilde u_j\approx v_j$. Let us now be more precise. For any two indices $i\neq j \in \{1,\, 2 ,\, \dots \,,\, k\}$, by~(\ref{stiminf}) and~(\ref{lipbound}) we have
\begin{equation}\label{ortoL^2}\begin{split}
\bigg|\int_E\tilde u_i\tilde u_j\bigg|&=\bigg|\int_E \bigg(v_i -\frac{v_i}{v_1}\,\varphi\bigg)\bigg(v_j -\frac{v_j}{v_1}\,\varphi\bigg)\bigg|
=\bigg|\int_B v_i v_j - 2\int_B \frac{v_iv_j}{v_1}\,\varphi + \int_B \frac{v_iv_j}{v_1^2}\, \varphi^2\bigg|\\
&\leq 2CK \|\varphi\|_{L^1(B)} + K^2 \|\varphi\|_{L^2(B)}^2\leq C\delta\,.
\end{split}\end{equation}
Moreover, we notice that
\begin{equation}\label{stifrak}
\bigg\|D\bigg( \frac{v_i}{v_1}\, \varphi\bigg)\bigg\|_{L^2(B)} \leq \bigg\|\frac{v_i}{v_1}\bigg\|_{W^{1,\infty}(B)} \, \|\varphi\|_{H^1(B)}\leq K\delta\,,
\end{equation}
then the very same calculation as in~(\ref{ortoL^2}) give
\begin{equation}\label{ortoH^1}
\bigg|\int_E D\tilde u_i\cdot D\tilde u_j\bigg| \leq C\delta\,.
\end{equation}
Now, keep in mind that our goal is to prove~(\ref{estcontenuto1}); hence, recalling~(\ref{dario}) and up to increase the final constant $C$, we can assume that $\lambda_1(E)-\lambda_1(B)$ is as small as we wish, hence that $\delta$ is as small as we wish by~(\ref{stimadelta}). As a consequence, by a standard argument (identical to what we already did in Lemma~\ref{lastone}), the estimates~(\ref{ortoL^2}) and~(\ref{ortoH^1}) together with~(\ref{deftildeuj}) and~(\ref{stf}) imply that the subspace of $H^1_0(E)$ generated by the functions $\{\tilde u_1,\, \tilde u_2,\, \dots\, , \, \tilde u_k\}$ has dimension $k$, hence recalling also~(\ref{stifrak}) we obtain $\lambda_k(E)\leq \lambda_k(B)+C\delta$. And finally, by~(\ref{stimadelta}) this estimate gives~(\ref{estcontenuto1}).
\end{proof}

\begin{remark}\label{obviousrem}
Obviously, if a set $E$ is contained in the ball $B_c$ for some constant $c$ with $|c|<1/2$, then~(\ref{estcontenuto1}) is still true for $E$ and $B_c$, up to multiply the constant $C$ by $4$: to prove that, it is enough to recall the scaling property~(\ref{scaling}) for the eigenvalues.
\end{remark}

\begin{lemma}\label{newone}
For every $\alpha<1/(3\wedge N)$ there exist $n\in\N$ and a constant $C$ such that, for every set $\Omega\subseteq\R^N$ of unit measure with $\lambda_1(\Omega)\leq \lambda_1(B)+1$, calling $\eps=|\Omega\Delta B|$ one has
\begin{align}\label{afterlaus}
\int_{\Omega\setminus B_{n\eps^\alpha}} |Du_1|^2 \leq C\eps^\alpha\,, && \int_{\Omega\setminus B_{n\eps^\alpha}} u_1^2\leq C\eps^{1+\alpha(3-N)^+}\,.
\end{align}
\end{lemma}
\begin{proof}
Let us first consider the case $N\geq 3$. Then, we can apply Corollary~\ref{maincor} with $k=1$, finding $n$ and $C$. Thanks to~(\ref{enough}), we already have the first estimate in~(\ref{afterlaus}). Concerning the second one, for $N\geq 3$ we only have to show that $\int_{\Omega\setminus B_{n\eps^\alpha}} u_1^2\leq C\eps$, and in fact this follows immediately from the uniform $L^\infty$ estimate~(\ref{stiminf}) and the fact that $|\Omega\setminus B_{n\eps^\alpha}|\leq \eps$.\par

Let us now consider that case $N=2$. We start again by applying Corollary~\ref{maincor} with $k=1$, and again we obtain $n$, $C$, and some $\bar t<n\eps^\alpha$ such that~(\ref{finalesti}) and~(\ref{enough}) hold; the latter implies again the first estimate in~(\ref{afterlaus}). The second one, this time, is more complicate, since we have to show $\int_{\Omega\setminus B_{n\eps^\alpha}} u_1^2\leq C\eps^{1+\alpha}$, which is no more obvious from~(\ref{stiminf}). Recall that $\H^1(S_{\bar t})\leq C\eps^{1-\alpha}$, since $\bar t$ has been provided by Lemma~\ref{lemmagen}, thus~(\ref{propgamma}) holds. We now argue for simplicity under the assumption that $S_{\bar t}$ is connected, hence a connected arc: in the general case, it suffices to repeat the same argument for any connected component.\par

Let $S$ be the segment joining the two endpoints of $S_{\bar t}$, and let $A$ be the set whose boundary is given by $S_{\bar t}\cup S$. We define now the set $E=A \cup (\Omega\setminus B_{\bar t})$, and we introduce the following function $v\in H^1(E)$: assuming without loss of generality that $S$ is a horizontal segment, for every $(x,y)\in A$ there exists a unique $z$ such that $(x,z)\in S_{\bar t}$, and we let
\[
v(x,y):=\left\{\begin{array}{ll}
u_1(x,y) \qquad &\hbox{if $(x,y)\in \Omega\setminus B_{\bar t}$}\,,\\
u_1(x,z) \qquad &\hbox{if $(x,y)\in A$}\,.
\end{array}\right.
\]
Since the length of $S$ is at most $C\eps^{1-\alpha}$, thus the distance between $S$ and $S_{\bar t}$ is at most $C\eps^{2-2\alpha}$, so we can easily estimate
\begin{equation}\label{brnd}
\int_E |Dv|^2 \leq \int_{\Omega\setminus B_{\bar t}} |Du_1|^2 + C\eps^{2-2\alpha} \int_{S_{\bar t}} |Du|^2 \leq C\eps^\alpha + C\eps^{2-2\alpha} \leq C\eps^\alpha\,.
\end{equation}
In addition, we have $|E|= |\Omega\setminus B_{\bar t}| + |A|\leq \eps + |S|^3\leq C\eps$. Notice now that $v$ belongs to $H^1(E)$, not to $H^1_0(E)$; nevertheless, since $E$ is entirely on one of the two parts in which $\R^2$ is divided by the line containing $S$, we can let $E^+$ be the union of $E$ with its symmetrical copy with respect to this line, and we can extend $v$ on $E^+\setminus E$ by symmetry; in this way, the resulting function $v$ belongs by construction to $H^1_0(E^+)$. Since $|E^+|=2|E|\leq C\eps$, by~(\ref{scaling}) we have $\lambda_1(E^+)\geq (C\eps)^{-1}$. And in turn, by~(\ref{brnd}) this implies
\[
\int_{\Omega\setminus B_{n\eps^\alpha}} u_1^2 \leq \int_{\Omega\setminus B_{\bar t}} u_1^2 \leq \int_E v^2 = \frac 12 \int_{E^+} v^2
\leq \frac {C\eps}2\, \int_{E^+} |Dv|^2 
=  C\eps \int_E |Dv|^2\leq C\eps^{1+\alpha}\,,
\]
which is the second inequality in~(\ref{afterlaus}). The proof is then concluded also in the planar case.
\end{proof}

\begin{corollary}\label{corollast}
Let $\alpha,\,n,\,\Omega$ and $\eps$ be as in Lemma~\ref{newone}. Calling $\widetilde\Omega=\Omega\cap B_{(n+1)\eps^\alpha}$, one has
\begin{equation}\label{newcor}
\lambda_1(\widetilde\Omega) \leq \lambda_1(\Omega)+ C\eps^\alpha\,.
\end{equation}
\end{corollary}
\begin{proof}
Let us define the function $\rho:\R^+\to [0,1]$ as
\[
\rho(r):=
\left\{\begin{array}{ll}
1\,, &\hbox{if } r\leq \omega_N^{-1/N}+n\eps^\alpha\,,\\
1-\frac{r-\omega_N^{-1/N}-n\eps^\alpha}{\eps^\alpha}\,,\qquad &\hbox{if }\omega_N^{-1/N}+n\eps^\alpha< r\leq \omega_N^{-1/N}+(n+1)\eps^{\alpha}\,,\\
0\,,&\hbox{if }r>\omega_N^{-1/N}+(n+1)\eps^\alpha\,,
\end{array}\right.
\]
and let us define $\tilde u(x)=u_1(x) \rho(|x|) \in H^1_0(\widetilde\Omega)$. By~(\ref{afterlaus}) and the fact that $\alpha<1/(3\wedge N)$, it is
\[\begin{split}
\int_{\widetilde\Omega} |D\tilde u|^2 &= \int_{\Omega\cap B_{n\eps^\alpha}} |Du_1|^2 + \int_{\widetilde\Omega\setminus B_{n\eps^\alpha}} \bigg(\rho Du_1 -  u_1 \eps^{-\alpha} \frac x{|x|}\bigg)^2\\
&\leq \int_\Omega |Du_1|^2 + \eps^{-2\alpha} \int_{\Omega\setminus B_{n\eps^\alpha}} u_1^2 + \eps^{-\alpha} \int_{\Omega\setminus B_{n\eps^\alpha}} u_1 |Du_1|\\
&\leq \lambda_1(\Omega) + C \eps^{1+\alpha((3-N)^+-2)} + C\eps^{\frac{1+\alpha((3-N)^+-1)}2}\leq \lambda_1(\Omega) + C\eps^\alpha\,.
\end{split}\]
Since, on the other hand, $\int_{\widetilde\Omega} \tilde u^2\geq \int_B u_1^2\geq 1-C\eps$, we conclude the validity of~(\ref{newcor}).
\end{proof}

We can then easily conclude the proof of Theorem~\ref{finalmain}.
\begin{proof}[Proof of Theorem~\ref{finalmain}, second part]
We only have to prove the estimate of $\lambda_k(\Omega)-\lambda_k(B)$ from above. Let $\alpha<1/(3\wedge N)$, and let $\Omega\subseteq\R^N$ be any set with unit measure such that $\lambda_1(\Omega)\leq \lambda_1(B)+1$. Up to a translation, we can assume that $\eps=|\Omega\Delta B|=d(\Omega)$, so that by~(\ref{bdv}) we have that $\eps\leq C\sqrt{\lambda_1(\Omega)-\lambda_1(B)}$. Let us also call $\widetilde\Omega =\Omega\cap{B_{(n+1) \eps^\alpha}}$, where $n$ is given by Lemma~\ref{newone}. By Lemma~\ref{leminclus} (together with Remark~\ref{obviousrem}) and Corollary~\ref{corollast}, and since $\widetilde\Omega\subseteq B_{(n+1)\eps^\alpha}$, we have
\[\begin{split}
\lambda_k(\Omega)-\lambda_k(B) &\leq \lambda_k(\widetilde\Omega) - \lambda_k(B_{(n+1)\eps^\alpha})
\leq C \Big(\lambda_1(\widetilde\Omega) - \lambda_1(B_{(n+1)\eps^\alpha})\Big)^{1/2}\\
&\leq C \Big(\lambda_1(\Omega)  - \lambda_1(B)+C\eps^\alpha\Big)^{1/2} \leq C\big(\lambda_1(\Omega)-\lambda_1(B)\big)^{\alpha/4}\,,
\end{split}\]
which gives the right inequality in~(\ref{mainest}) and~(\ref{mainest'}). The proof is then concluded.
\end{proof}

We conclude with a brief comments on the sharp exponents. We do not believe that the exponents found in~(\ref{mainest}) and~(\ref{estfraenkel}), and in~(\ref{mainest'}) and~(\ref{estfraenkel'}), are sharp. In particular, as already observed at the end of Section~\ref{sec:prel}, for $k=2$ the right inequality in~(\ref{mainest}) and~(\ref{mainest'}) holds with exponent $1$. Nonetheless, as far as we know Theorem~\ref{finalmain} is the first result where it is provided a quantitative estimate for the eigher eigenvalues of sets close to the ball with exponents that do \emph{not} depend on the dimension.

\appendix
\section*{Acknowledgments.}
This work originated from a question suggested to us by Antoine Henrot. The first author has been supported by the ERC advanced grant n.~339958 ``COMPAT''.

\end{document}